\documentclass[reqno]{amsart}
\usepackage{hyperref}

\begin{document}

\title[\hfilneg \hfil Unicity on entire function with respect to its differential-difference polynomials in several complex variables ]
{Unicity on entire function with respect to its differential-difference polynomials several complex variables}

\author[XiaoHuang Huang \hfil \hfilneg]
{XiaoHuang Huang}

\address{XiaoHuang Huang: Corresponding author\newline
Department of Mathematics, Southern University of Science and Technology, Shenzhen 518055, China}
\email{1838394005@qq.com}

\begin{abstract}
In this paper, we study the uniqueness of the differential-difference polynomials of entire functions on $\mathbb{C}^{n}$. We prove the following result: Let $f(z)$  be a  transcendental entire function on $\mathbb{C}^{n}$ of hyper-order less than $1$ and $g(z)=b_{-1}+\sum_{i=0}^{n}b_{i}f^{(k_{i})}(z+\eta_{i})$, where $b_{-1}$ and $b_{i} (i=0\ldots,n)$ are small meromorphic functions of $f$ on $\mathbb{C}^{n}$, $k_{i}\geq0 (i=0\ldots,n)$ are integers, and $\eta_{i} (i=0\ldots,n)$ are finite values.  Let $a_{1}(z)\not\equiv\infty, a_{2}(z)\not\equiv\infty$ be two distinct small meromorphic functions of $f(z)$ on $\mathbb{C}^{n}$. If $f(z)$ and $g(z)$ share $a_{1}(z)$ CM, and  $a_{2}(z)$ IM. Then either $f(z)\equiv g(z)$ or $a_{1}=2a_{2}=2$,
$$f(z)\equiv e^{2p}-2e^{p}+2,$$
and
$$g(z)\equiv e^{p},$$
where $p(z)$ is a non-constant entire function on $\mathbb{C}^{n}$. Especially, in the case of $g(z)=(\Delta_{\eta}^{n}f(z))^{k}$, we obtain $f(z)\equiv (\Delta_{\eta}^{n}f(z))^{k}$.
\end{abstract}

\maketitle
\numberwithin{equation}{section}
\newtheorem{theorem}{Theorem}[section]
\newtheorem{lemma}[theorem]{Lemma}
\newtheorem{remark}[theorem]{Remark}
\newtheorem{corollary}[theorem]{Corollary}
\newtheorem{example}[theorem]{Example}
\newtheorem{problem}[theorem]{Problem}
\allowdisplaybreaks

\section{Introduction }
In this paper, we assume that the reader is familiar with the basic notations of Nevanlinna's value distribution theory, see \cite{h3,y1,y2}. In the following, a meromorphic function $f(z)$ means meromorphic on $\mathbb{C}^{n}, n\in \mathbb{N}^{+}$. By $S(r,f)$, we denote any quantity satisfying $S(r,f)=o(T(r,f))$ as $r\rightarrow\infty$, outside of an exceptional set of finite linear or logarithmic measure.

Let $a$ be a complex numbers. We say that two nonconstant meromorphic functions $f(z)$ and $g(z)$ share value $a$ IM (CM) if $f(z)-a$ and $g(z)-a$ have the same zeros ignoring multiplicities (counting multiplicities).

For a given meromorphic function $f: \mathbb{C}^{n}\rightarrow\mathbb{P}^{1}$ and nonzero vector $\eta=(\eta^{1}, \eta^{2},\ldots,\eta^{n})\in\mathbb{C}^{n}\backslash {0}$, we define the shift by $f(z+\eta)$ and the difference operators by
$$\Delta_{\eta}f(z)=f(z^{1}+\eta^{1},\ldots,z^{n}+\eta^{n})-f(z^{1},\ldots,z^{n}),$$
$$\Delta_{\eta}f(z)=\Delta_{\eta}(\Delta_{\eta}^{n-1}f(z)),\quad n\in\mathbb{N},n\geq2,$$
where $z=(z^{1},\ldots,z^{n})\in\mathbb{C}^{n}$.

Suppose $|z|=(|z^{1}|^{2}+|z^{2}|^{2}+\cdots|z^{n}|^{2})^{\frac{1}{2}}$ for $z=(z^{1},z^{1},\ldots,z^{n})\in \mathbb{C}^{n}$. For $r>0$, denote
$$B_{n}(r):={z\in \mathbb{C}^{n}||z|<r}, \quad S_{n}(r):={z\in \mathbb{C}^{n}||z|=r}.$$
Let $d=\partial+\overline{\partial}$, $d^{c}=(4\pi\sqrt{-1})^{-1}(\partial-\overline{\partial})$. Then $dd^{c}=\frac{\sqrt{-1}}{2\pi}\partial\overline{\partial}$. We write
$$\omega_{n}(z):=(dd^{c}log|z|^{2}),\quad \sigma_{n}(z):=d^{c}log|z|^{2}\Lambda\omega_{n}^{n-1}(z),$$
for $z\in \mathbb{C}^{n}$ a nonzero complex number.
$$\upsilon_{n}(z)=dd^{c}|z|^{2}, \quad \rho_{n}(z)=\upsilon_{n}^{n}(z),$$
for $z\in\mathbb{C}$.

Thus $\sigma_{n}(z)$ defines a positive measure on $S_{n}(r)$ with total measure one and $\rho_{n}(z)$ is Lebesgue measure on $\mathbb{C}^{n}$ normalized such that $B_{n}(r)$ has measure $r^{2n}$. Moreover, when we restrict $\upsilon_{n}(z)$ to $S_{n}(r)$, we obtain that
$$\upsilon_{n}(z)=r^{2}\omega_{n}(z)\quad and \quad \int_{B_{n}(r)}\omega_{n}^{n}=1.$$

Let $f$ be a meromorphic function on $\mathbb{C}^{n}$, i.e., $f$ can be written as a quotient of two holomorphic functions which are relatively prime. Thus $f$ can be regarded as a meromorphic map $f: \mathbb{C}^{n}\rightarrow\mathbb{P}^{1}$ such that $f^{-1}(\infty)\neq\mathbb{C}^{n}$; i.i. $f(z)=[f_{0}(z),f_{1}(z)]$ and $f_{0}$  is not identity equal to zero. Clearly the meromorphic map $f$ is not defined on the set $I_{f}\{z\in\mathbb{C}^{n}; f_{0}(z)=f_{1}(z)=0\}$, which is called the set of indeterminacy of $f$, and $I_{f}$ is an analytic subvariety of $\mathbb{C}^{n}$ with codimension not less than $2$. Thus we can define, for $z \in \mathbb{C}^{n}\backslash I_{f}$,
$$f^{*}\omega=dd^{c}log(|f_{0}|^{2}+|f_{1}|^{2}),$$
where $\omega$ is the Fubini-Study form. Therefore, for any measurable set $X\subset \mathbb{C}^{n}$, integrations of $f$ over $X$ may be defined as integrations over $X\backslash I_{f}$.

For all $0<s<r$, the characteristic function of $f$ is defined by
$$T_{f}(r,s)=\int_{s}^{r}\frac{1}{t^{2n-1}}\int_{B_{n}(t)}f^{*}(\omega)\Lambda\omega_{n}^{n-1}dt.$$

Let $a\in \mathbb{P}^{1}$ with $f^{-1}(a)\neq\mathbb{C}^{n}$ and $Z_{a}^{f}$ be an $a-divisor$ of $f$. We write $Z_{a}^{f}(t)=\overline{\mathbb{B}}_{n}(t)\bigcap Z_{a}^{f}$. Then the pre-counting function and counting function with respect to $a$ are defined, respectively, as (if $0\not\in Z_{a}^{f}$)
$$n_{f}(t,a)=\int_{Z_{a}^{f}(t)\omega^{n-1}} \quad and \quad N_{f}(r,a)=\int_{0}^{r}n_{f}(t,a)\frac{dt}{t}.$$
Therefore Jensen's formula is, if $f(0)\neq0$, for all $r\in \mathbb{R}^{+}$,
$$N_{f}(r,0)-N_{f}(r,\infty)=\int_{S_{n}(r)}log|f(z)|\sigma_{n}(z)-loglog|f(0)|.$$

Let $a\in \mathbb{P}^{1}$ with $f^{-1}(a)\neq\mathbb{C}^{n}$, then we define the proximity function as
\begin{eqnarray*}
\begin{aligned}
m_{f}(r,a)&=\int_{S_{n}(r)}log^{+}\frac{1}{|f(z)-a|}\sigma_{n}(z), if \quad a\neq\infty;\\
&=\int_{S_{n}(r)}log^{+}|f(z)|\sigma_{n}(z), if \quad a=\infty.
\end{aligned}
\end{eqnarray*}
The first main theorem states that, if $f(0)\neq a,\infty$,
$$T_{f}(r,s)=N_{f}(r,s)+m_{f}(r,s)-log\frac{1}{|f(z)-a|}$$
where $0<s<r$.

In this paper, we write $N(r,f):=N_{f}(r,\infty)$, $N(r,\frac{1}{f}):=N_{f}(r,0)$, $m_{f}(r,0):=m(r,\frac{1}{f})$, $m_{f}(r,\infty):=m(r,f)$ and $T_{f}(r,s)=T(r,f)$. Hence $T(r,f)=m(r,f)+N(r,f)$. And we can deduce the First Fundamental Theorem of Nevanlinna on $\mathbb{C}^{n}$
\begin{align}
T(r,f)=T(r,\frac{1}{f-a})+O(1).
\end{align}
More details can be seen in \cite{r,y3}.

Furthermore, meromorphic functions $f$ on $\mathbb{C}^{n}$, we define
 $$\rho(f)=\varlimsup_{r\rightarrow\infty}\frac{log^{+}T(r,f)}{logr},$$
 $$\rho_{2}(f)=\varlimsup_{r\rightarrow\infty}\frac{log^{+}log^{+}T(r,f)}{logr}$$
by the order  and the hyper-order  of $f$, respectively.

A meromorphic function $f$ satisfying the condition
$$\varlimsup_{r\rightarrow\infty}\frac{logT(r,f)}{r}=0,$$
of above is said to be a meromorphic function with $\rho_{2}(f)<1$.

\section{Main results}

In 1977, Rubel and Yang \cite{ruy}  considered the uniqueness of an entire function and its derivative. They proved.

\

{\bf Theorem A} \ Let $f(z)$ be a transcendental entire function, and let $a, b$ be two finite distinct complex values. If $f(z)$ and $f'(z)$
 share $a, b$ CM, then $f(z)\equiv f'(z)$.

During 2006-2008, the difference analogue of the lemma on the logarithmic derivative and Nevanlinna theory for the difference operator have been founded, which bring about a number of papers $[3-8, 15-17]$ focusing on the uniqueness study of meromorphic functions sharing some values with their difference operators. Heittokangas et al \cite{hkl} obtained a similar result analogue of Theorem A concerning shifts.

\

{\bf Theorem B}
 Let $f(z)$ be a non-constant entire function of finite order, let $c$ be a nonzero finite complex value, and let $a, b$ be two finite distinct complex values.
If $f(z)$ and $f(z+c)$ share $a, b$ CM, then $f(z)\equiv f(z+c).$

With the establishment of logarithmic derivative lemma in several variables by A.Vitter \cite{v} in 1977, a number of papers about Nevanlinna Theory in several variables were published \cite{hy2,hy3,y3}. In 1996, Hu-Yang \cite{hy2} generalized  Theorem 1 in the case of higher dimension. They proved.

\

{\bf Theorem C} \ Let $f(z)$ be a transcendental entire function on $\mathbb{C}^{n}$, and let $a, b$$\in \mathbb{C}^{n}$ be two finite distinct complex values. If $f(z)$ and $D_{u}f(z)$ share $a, b$ CM, then $f(z)\equiv D_{u}f(z)$, where $D_{u}f(z)$ is a directional derivative of $f(z)$ along a direction $u\in S^{2n-1}$.

In recent years, there has been tremendous interests in developing  the value distribution of meromorphic functions with respect to difference analogue,  see [1-3, 5-10, 12-17, 21].
 Heittokangas et al \cite{hkl} proved a similar result analogue of Theorem A concerning shift.

\

{\bf Theorem D}
 Let $f(z)$ be a non-constant entire function of finite order, let $\eta$ be a nonzero finite complex value, and let $a_{1}, a_{2}$ be two finite distinct complex values. If $f(z)$ and $f(z+\eta)$ share $a_{1}, a_{2}$ CM, then $f(z)\equiv f(z+\eta).$

In 2014,  Liu-Yang-Fang \cite{lyf}  proved

\

{\bf Theorem E}
 Let $f$ be a transcendental entire function of finite order,  let $\eta$ be a non-zero complex number, $n$ be a positive integer,  and let $a_{1}, a_{2}$ be two finite distinct complex values. If $ f$ and $\Delta_{\eta}^{n}f$ share $a_{1}, a_{2}$ CM, then $ f\equiv \Delta_{\eta}^{n}f$.

Recently, we   proved

\

{\bf Theorem F}
 Let $f(z)$ be a  transcendental entire function of $\rho_{2}(f)<1$, let $\eta\neq0$ be a finite complex number, $n\geq1, k\geq0$  two  integers and let $a, b$ be two  distinct finite complex values. If $f(z)$ and $(\Delta_{\eta}^{n}f(z))^{(k)}$ share $a_{1}$ CM and  $a_{2}$ IM, then $f(z)\equiv(\Delta_{\eta}^{n}f(z))^{(k)}$.

In recent years, there has been tremendous interests in developing  the value distribution of meromorphic functions with respect to difference analogue in the case of higher dimension. Especially in 2020, Cao-Xu \cite{cx} established the difference analogue of the lemma in several variables, one can study some interesting uniqueness problems on meromorphic functions sharing values with their shift or difference operators corresponding to the uniqueness problems on meromorphic functions sharing values with their derivatives in several variables.

In this paper, we continuous to investigate Theorem F, and  obtain.
\

{\bf Theorem 1}
 Let $f(z)$ be a non-constant meromorphic function with $\rho_{2}(f)<1$ on $\mathbb{C}^{n}$ and $g(z)=b_{-1}+\sum_{i=0}^{n}b_{i}f^{(k_{i})}(z+\eta_{i})$, where $b_{-1}$ and $b_{i} (i=0\ldots,n)$ are small meromorphic functions of $f(z)$ on $\mathbb{C}^{n}$, $k_{i}\geq0 (i=0\ldots,n)$ are integers, and $\eta_{i} (i=0\ldots,n)$ are finite values.  Let $a_{1}(z)\not\equiv\infty, a_{2}(z)\not\equiv\infty$ be two distinct small meromorphic functions of $f(z)$ on $\mathbb{C}^{n}$. If $f(z)$ and $g(z)$ share $a_{1}(z)$ CM, and  $a_{2}(z)$ IM. Then either $f(z)\equiv g(z)$ or $a_{1}=2a_{2}=2$,
$$f(z)\equiv e^{2p}-2e^{p}+2,$$
and
$$g(z)\equiv e^{p},$$
where $p(z)$ is a non-constant entire function on $\mathbb{C}^{n}$.

\

{\bf Corollary 1}
Let $f(z)$ be a non-constant entire function  of $\rho_{2}(f)<1$ on $\mathbb{C}^{n}$, and let $a_{1}$ and $a_{2}$ be two distinct small functions of $f$ on $\mathbb{C}^{n}$. If $f(z)$ and $(\Delta_{\eta}^{n} f(z))^{(k)}$ share $a_{1}(z)$ CM, and  $a_{2}(z)$ IM. Then either $f(z)\equiv (\Delta_{\eta}^{n} f(z))^{(k)}$ .

\

{\bf Remark 1} For convenience, through out the paper, $o(T(r,f))$ always means that it holds for all $r\not\in E$ with
$$\overline{dens}E=lim_{r\rightarrow\infty}sup\frac{1}{r}\int_{E\bigcap [1,r]}dt=0.$$
We will not repeat a long sentence as above.  

\section{Some Lemmas}
\begin{lemma}\label{21l} \cite {cx}
Let $f$ be a nonconstant meromorphic function on $\mathbb{C}^{n}$, let $\eta\in \mathbb{C}^{n}$ be a nonzero finite complex number. If
$$\varlimsup_{r\rightarrow\infty}\frac{logT(r,f)}{r}=0,$$
then
$$m(r,\frac{f(z+\eta)}{f(z)})+m(r,\frac{f(z)}{f(z+\eta)})=o(T(r,f)).$$
\end{lemma}

\begin{lemma}\label{221}\cite{v} Let $f(z)$ is a non-constant meromorphic function on $\mathbb{C}^{n}$, and let $\upsilon=(\upsilon_{1},\ldots,\upsilon_{n})\in Z_{+}^{n}$ be a multi-index. Then for any $\varepsilon>0$,
$$m(r,\frac{\partial^{\upsilon}f}{f})\leq|\upsilon|log^{+}|T(r,f)|+|\upsilon|log^{+}|T(r,f)|+O(1)=o(T(r,f)).$$
\end{lemma}

\begin{lemma}\label{21l} \cite {cx}
Let $f$ be a nonconstant meromorphic function on $\mathbb{C}^{n}$, let $\eta\in \mathbb{C}^{n}$ be a nonzero finite complex number. If
$$\varlimsup_{r\rightarrow\infty}\frac{logT(r,f)}{r}=0,$$
then
$$T(r,f(z+\eta))=T(r,f(z))+o(T(r,f))$$
and
$$N(r,f(z+\eta))=N(r,f(z))+o(T(r,f)).$$
\end{lemma}

\begin{lemma}\label{23l} Let $f(z)$  be a  transcendental entire function of $\rho_{2}(f)<1$ on $\mathbb{C}^{n}$ and $g(z)=b_{-1}+\sum_{i=0}^{n}b_{i}f^{(k_{i})}(z+\eta_{i})$, where $b_{-1}$ and $b_{i} (i=0\ldots,n)$ are small meromorphic functions of $f(z)$ on $\mathbb{C}^{n}$, $k_{i}\geq0 (i=0\ldots,n)$ are integers, and $\eta_{i} (i=0\ldots,n)$ are finite values.  Let $a_{1}(z)\not\equiv\infty, a_{2}(z)\not\equiv\infty$ be two distinct small meromorphic functions of $f(z)$ on $\mathbb{C}^{n}$ . Suppose
\[L(f)=\left|\begin{array}{rrrr}a_{1}-a_{2}& &f-a_{1} \\
a'_{1}-a'_{2}& &f'-a'_{1}\end{array}\right|\]
and
\[L(g)=\left|\begin{array}{rrrr}a_{1}-a_{2}& &g-a_{1} \\
a'_{1}-a'_{2}& &g'-a'_{1}\end{array}\right|,\]
and $f$ and $g$ share $a_{1}$ CM, and  $a_{2}$ IM,  then $L(f)\not\equiv0$ and $L(g)\not\equiv0$.
\end{lemma}
\begin{proof}
Suppose that $L(f)\equiv0$, then we can get $\frac{f'-a'_{1}}{f-a_{1}}\equiv\frac{a'_{1}-a'_{2}}{a_{1}-a_{2}}$. Integrating both side of above we can obtain $f-a=C_{1}(a_{1}-a_{2})$, where $C_{1}$ is a nonzero constant. So we have $T(r,f)=o(T(r,f))$, a contradiction. Hence $L(f)\not\equiv0$.

Since $g$ and $f$ share $a_{1}$ CM and $a_{2}$ IM, and $f$ is a transcendental entire function of $\rho_{2}(f)<1$, then by the  Second Fundamental Theorem of Nevanlinna, we get
\begin{eqnarray*}
\begin{aligned}
T(r,f)&\leq \overline{N}(r,\frac{1}{f-a_{1}})+\overline {N}(r,\frac{1}{f-a_{2}})+o(T(r,f))\notag\\
&= \overline {N}(r,\frac{1}{g-a_{1}})+\overline {N}(r,\frac{1}{g-a_{2}})+o(T(r,f))\notag\\
&\leq 2T(r,g)+o(T(r,f)).
\end{aligned}
\end{eqnarray*}
Hence $a_{1}$ and $a_{2}$ are small functions of $g$. If $L(g)\equiv0$, then we can get $g-a_{1}=C_{2}(a_{1}-a_{2})$, where $C_{2}$ is a nonzero constant. And we get $T(r,g)=o(T(r,f))$. Combing above inequality we obtain $T(r,f)=o(T(r,f))$, and hence a contradiction.
\end{proof}

\begin{lemma}\label{24l}  Let $f(z)$  be a  transcendental entire function of $\rho_{2}(f)<1$ on $\mathbb{C}^{n}$ and $g(z)=b_{-1}+\sum_{i=0}^{n}b_{i}f^{(k_{i})}(z+\eta_{i})$, where $b_{-1}$ and $b_{i} (i=0\ldots,n)$ are small meromorphic functions of $f$ on $\mathbb{C}^{n}$, $k_{i}\geq0 (i=0\ldots,n)$ are integers, and $\eta_{i} (i=0\ldots,n)$ are finite values.  Let $a_{1}(z)\not\equiv\infty, a_{2}(z)\not\equiv\infty$ be two distinct small meromorphic functions of $f(z)$ on $\mathbb{C}^{n}$.  Again let $d_{j}=a-l_{j}(a_{1}-a_{2})$ $(j=1,2,\ldots,q)$. Then
$$m(r,\frac{L(f)}{f-a_{1}})=S(r,f), \quad m(r,\frac{L(f)}{f-a_{2}})=S(r,f).$$
And
$$m(r,\frac{L(f)f}{(f-d_{1})(f-d_{2})\cdots(f-d_{m})})=o(T(r,f)),$$
where $L(f)$ is defined as in Lemma 2.3, and $2\leq m\leq q$.
\end{lemma}
\begin{proof}
Obviously, we have
$$m(r,\frac{L(f)}{f-a_{1}})\leq m(r,-\frac{(a'_{1}-a'_{2})(f-a_{1})}{f-a_{1}})+m(r,\frac{(a_{1}-a_{2})(f'-a'_{1})}{f-a_{1}})=o(T(r,f)).$$
And
$$\frac{L(f)f}{(f-a_{1})(f-a_{2})\cdots(f-a_{q})}=\sum_{i=1}^{q}\frac{C_{i}L(f)}{f-a_{i}},$$
where $C_{i}(i=1,2\ldots,q)$ are small functions of $f$. By Lemma 2.1 and above, we have
$$m(r,\frac{L(f)f}{(f-a_{1})(f-a_{2})\cdots(f-a_{q})})=m(r,\sum_{i=1}^{q}\frac{C_{i}L(f)}{f-a_{i}})\leq\sum_{i=1}^{q}m(r,\frac{L(f)}{f-a_{i}})=o(T(r,f)).$$
\end{proof}
\begin{lemma}\label{28l}\cite{y3} Let $f(z)$  be a non-constant meromorphic function of $\rho_{2}(f)<1$ on $\mathbb{C}^{n}$, and let $a_{1}$, $a_{2}$ and $a_{3}$ be three distinct small functions of $f$ on $\mathbb{C}^{n}$. Then
$$T(r,f)\leq \overline{N}(r,\frac{1}{f-a_{1}})+\overline{N}(r,\frac{1}{f-a_{2}})+\overline{N}(r,\frac{1}{f-a_{3}})+o(T(r,f)).$$
\end{lemma}

\begin{lemma}\label{26l} Let $f(z)$  be a  transcendental entire function of $\rho_{2}(f)<1$ on $\mathbb{C}^{n}$ and $g(z)=b_{-1}+\sum_{i=0}^{n}b_{i}f^{(k_{i})}(z+\eta_{i})$, where $b_{-1}$ and $b_{i} (i=0\ldots,n)$ are small meromorphic functions of $f$ on $\mathbb{C}^{n}$, $k_{i}\geq0 (i=0\ldots,n)$ are integers, and $\eta_{i} (i=0\ldots,n)$ are finite values.  Let $a_{1}(z)\not\equiv\infty, a_{2}(z)\not\equiv\infty$ be two distinct small meromorphic functions of $f(z)$ on $\mathbb{C}^{n}$. If $f(z)$ and $g(z)$ share $a_{1}$ CM, and 
$$N(r,\frac{1}{g(z)-(b_{-1}+\sum_{i=0}^{n}b_{i}a_{1}^{(k_{i})}(z+\eta_{i}))})=o(T(r,f)).$$ 
Then there is an entire function $p$ on $\mathbb{C}^{n}$ such that either $g=He^{p}+G$, where $G=b_{-1}+\sum_{i=0}^{n}b_{i}a_{1}^{(k_{i})}(z+\eta_{i})$, or $T(r,e^{p})=o(T(r,f))$.
\end{lemma}
\begin{proof}
 Since $f$ is a transcendental entire function of $\rho_{2}(f)<1$, and $f(z)$ and $g(z)$ share $a_{1}$ CM, then there is entire function $p$ such that
\begin{align}
f-a_{1}=Ae^{p}(g-G)+Ae^{p}(G-a_{1}),
\end{align}
where the zeros and poles of $A$ come from the zeros and poles of $b_{-1}$ and $b_{i}(i=0,1,\ldots,n)$, and  $G=b_{-1}+\sum_{i=0}^{n}b_{i}a_{1}^{(k_{i})}(z+\eta_{i})$.

Suppose that $T(r,e^{p})\neq o(T(r,f))$. Set $Q=g-G$. Do induction from (3.1) that
\begin{align}
Q=\sum_{i=0}^{n}b_{i}(A_{i\eta}e^{p_{i\eta}}Q_{i\eta})^{(k_{i})}+\sum_{i=0}^{n}b_{i}(A_{i\eta}e^{p_{i\eta}}(G-a_{1})_{i\eta})^{(k_{i})}+G.
\end{align}
Easy to see that $Q\not\equiv0$. Then we rewrite (3.2) as
\begin{eqnarray}
1-\frac{\sum_{i=0}^{n}b_{i}(A_{i\eta}e^{p_{i\eta}}(G-a_{1})_{i\eta})^{(k_{i})}+G}{Q}=De^{p},
\end{eqnarray}
where
\begin{align}
D&=\frac{\sum_{i=0}^{n}b_{i}(A_{i\eta}e^{p_{i\eta}}Q_{i\eta})^{(k_{i})}}{Qe^{p}}
\end{align}
Note that $N(r,\frac{1}{g-G})=N(r,\frac{1}{Q})=o(T(r,f))$. Then
\begin{align}
T(r,D)&\leq\sum_{i=0}^{n}(T(r,\frac{(A_{i\eta}e^{p_{i\eta}}Q_{i\eta})^{(k_{i})}}{Qe^{p}})+o(T(r,f))\notag\\
&\leq m(r,\frac{(A_{i\eta}e^{p_{i\eta}}Q_{i\eta})^{(k_{i})}}{Qe^{p}})+N(r,\frac{(A_{i\eta}e^{p_{i\eta}}Q_{i\eta})^{(k_{i})}}{Qe^{p}})+S(r,e^{p})+o(T(r,f))\notag\\
&=S(r,e^{p})+o(T(r,f)).
\end{align}
By (3.1) and Lemma 3.1,  we get
\begin{align}
T(r,e^{p})&\leq T(r,f)+T(r,g)+o(T(r,f))\notag\\
&\leq 2T(r,f)+o(T(r,f)),
\end{align}
then it follows from (3.5) that $T(r,D)=o(T(r,f))$.

Next we discuss  two cases.

{\bf Case1.} \quad $e^{-p}-D\not\equiv0$. Rewrite (3.3) as
\begin{align}
Qe^{p}(e^{-p}-D)=\sum_{i=0}^{n}b_{i}(A_{i\eta}e^{p_{i\eta}}(G-a_{1})_{i\eta})^{(k_{i})}+G.
\end{align}
We claim that $D\equiv0$. Otherwise, it follows from (2.8) that $N(r,\frac{1}{e^{-p}-D})=o(T(r,f))$. Then use Lemma 3.5 to $e^{p}$ we can obtain
\begin{align}
T(r,e^{p})&=T(r,e^{-p})+O(1)\notag\\
&\leq \overline{N}(r,e^{-p})+\overline{N}(r,\frac{1}{e^{-p}})+\overline{N}(r,\frac{1}{e^{-p}-D})\notag\\
&+O(1)=o(T(r,f)),
\end{align}
and which contradicts with assumption. Thus $D\equiv0$. Then by (3.8) we get
\begin{align}
g=\sum_{i=0}^{n}b_{i}(A_{i\eta}e^{p_{i\eta}}(G-a_{1})_{i\eta})^{(k_{i})}+G=He^{p}+G,
\end{align}
where $H\not\equiv0$ is a small function of $e^{p}$.

{\bf Case2.} \quad $e^{-p}-D\equiv0$. Immediately, we get $T(r,e^{p})=o(T(r,f))$.
\end{proof}

\begin{lemma}\label{37l}\cite{hly}
Let $f$ be a nonconstant meromorphic function on $\mathbb{C}^{n}$, and $R(f)=\frac{P(f)}{Q(f)}$, where
$$P(f)=\sum_{k=0}^{p}a_{k}f^{k} \quad and \quad Q(f)=\sum_{j=0}^{q}a_{j}f^{q}$$
are two mutually prime polynomials in $f$. If the coefficients ${a_{k}}$ and ${b_{j}}$ are small functions of $f$ on $\mathbb{C}^{n}$ and $a_{p}\not\equiv0$, $b_{q}\not\equiv0$, then
$$T(r,R(f))=max\{p,q\}T(r,f)+o(T(r,f)).$$
\end{lemma}

\begin{lemma}\label{381}   Let $f(z)$  be a  transcendental entire function of $\rho_{2}(f)<1$ on $\mathbb{C}^{n}$ and $g(z)=b_{-1}+\sum_{i=0}^{n}b_{i}f^{(k_{i})}(z+\eta_{i})$, where $b_{-1}$ and $b_{i} (i=0\ldots,n)$ are small meromorphic functions of $f$ on $\mathbb{C}^{n}$, $k_{i}\geq0 (i=0\ldots,n)$ are integers, and $\eta_{i} (i=0\ldots,n)$ are finite values.  Let $a_{1}(z)\not\equiv\infty, a_{2}(z)\not\equiv\infty$ be two distinct small meromorphic functions of $f(z)$ on $\mathbb{C}^{n}$. If $f(z)$ and $g(z)$ share $a_{1}(z)$ CM, and  $a_{2}(z)$ IM, and if $f\not\equiv g$, then\\
(i)\quad $T(r,f)=\overline{N}(r,\frac{1}{f-a_{1}})+\overline{N}(r,\frac{1}{f-a_{2}})+o(T(r,f))$.\\
(ii)\quad $m(r,\frac{1}{f})=o(T(r,f))$.\\
(iii)\quad $N(r,\frac{1}{f-a_{1}})=\overline{N}(r,\frac{1}{f-a_{1}})+o(T(r,f))$\\
(iv)\quad $\overline{N}(r,\frac{1}{f-a_{2}})=m(r,\frac{1}{f-a_{1}})+o(T(r,f))$.
\end{lemma}
\begin{proof}
If $f\equiv g$, there is nothing to prove. Suppose $f\not\equiv g$. Since $f$ is a transcendental entire function of $\rho_{2}(f)<1$,  $f$ and $g$ share $a_{1}$ CM, then  we get
\begin{align}
\frac{g-a_{1}}{f-a_{1}}=Be^{h},
\end{align}
where $h$ is entire function, and (3.1) implies $h=-p$ and $B=\frac{1}{A}$.\\

Since $f$ and $g$ share $a_{1}$ CM and share $a_{2}$ IM, then by  Lemma 3.1 and Lemma 3.5 we have
\begin{eqnarray*}
\begin{aligned}
T(r,f)&\leq \overline{N}(r,\frac{1}{f-a_{1}})+\overline{N}(r,\frac{1}{f-a_{2}})+o(T(r,f))= \overline{N}(r,\frac{1}{g-a_{1}})\\
&+\overline{N}(r,\frac{1}{g-a_{1}})+o(T(r,f))\leq N(r,\frac{1}{f-g})+o(T(r,f))\\
&\leq T(r,f-g)+S(r,f)\leq m(r,f-g)+o(T(r,f))\\
&= m(r,f-\sum_{i=0}^{n}b_{i}f^{(k_{i})}_{i\eta})+o(T(r,f))\\
&\leq m(r,f)+m(r,1-\frac{\sum_{i=0}^{n}b_{i}f^{(k_{i})}_{i\eta}}{f})+\leq T(r,f)+o(T(r,f)).
\end{aligned}
\end{eqnarray*}
That is
\begin{eqnarray}
T(r,f)=\overline{N}(r,\frac{1}{f-a_{1}})+\overline{N}(r,\frac{1}{f-a_{2}})+o(T(r,f)).
\end{eqnarray}
According to (3.11) we have
\begin{eqnarray}
T(r,f)=T(r,f-g)+S(r,f)=N(r,\frac{1}{f-g})+o(T(r,f)),
\end{eqnarray}
and
\begin{align}
T(r,Be^{h})&=m(r,Be^{h})+o(T(r,f))=m(r,\frac{g-G+G-a_{1}}{f-a_{1}})+o(T(r,f))\notag\\
&\leq m(r,\frac{1}{f-a_{1}})+o(T(r,f)),
\end{align}
where $G$ is defined as in Lemma 3.7.\\
Then  (3.11) and (3.13) deduce that
\begin{align}
m(r,\frac{1}{f-a_{1}})&=m(r,\frac{Be^{h}-1}{f-g})\notag\\
&\leq m(r,\frac{1}{f-g})+m(r,Be^{h}-1)\notag\\
&\leq T(r,e^{h})+o(T(r,f)).
\end{align}
Then by (3.13) and (3.14)
\begin{align}
T(r,e^{h})= m(r,\frac{1}{f-a_{1}})+o(T(r,f)).
\end{align}
On the other hand, we rewrite (3.10) as follow
\begin{align}
\frac{g-f}{f-a_{1}}=Be^{h}-1,
\end{align}
and it follows that
\begin{align}
\overline{N}(r,\frac{1}{f-a_{2}})\leq \overline{N}(r,\frac{1}{Be^{h}-1})=T(r,e^{h})+o(T(r,f)).
\end{align}
Then by (3.11), (3.15) and (3.17)
\begin{eqnarray*}
\begin{aligned}
m(r,\frac{1}{f-a_{1}})+N(r,\frac{1}{f-a_{1}})&= \overline{N}(r,\frac{1}{f-a_{1}})+\overline{N}(r,\frac{1}{f-a_{2}})+o(T(r,f))\\
&\leq \overline{N}(r,\frac{1}{f-a_{1}})+\overline{N}(r,\frac{1}{Be^{h}-1})+o(T(r,f))\\
&\leq\overline{N}(r,\frac{1}{f-a_{1}})+m(r,\frac{1}{f-a_{1}})+o(T(r,f)).
\end{aligned}
\end{eqnarray*}
That is
\begin{align}
N(r,\frac{1}{f-a_{1}})=\overline{N}(r,\frac{1}{f-a_{1}})+o(T(r,f)),
\end{align}
and hence
\begin{align}
\overline{N}(r,\frac{1}{f-a_{2}})=T(r,e^{h})+o(T(r,f)).
\end{align}
\end{proof}

\begin{lemma}\label{28l}\cite{hly}
Let $f$ be a nonconstant meromorphic function on $\mathbb{C}^{n}$, and let $P(f)=a_{0}+a_{1}f+a_{2}f^{2}+\cdots+a_{n}f^{n}$, where $a_{i}$ are  small functions of $f$ for $i=0,1,\ldots,n$ on $\mathbb{C}^{n}$. Then
$$T(r,P(f))=nT(r,f)+S(r,f).$$
\end{lemma}

\begin{lemma}\label{291}\cite{y1} 
Suppose $f_{1}, f_{2},\cdots, f_{n}(n\neq2)$   are meromorphic functions on $\mathbb{C}^{n}$ and $g_{1}, g_{2},\cdots, g_{n}$ are entire functions on $\mathbb{C}^{n}$ such that\\
(i) $\sum_{j=1}^{n}f_{j}e^{g_{j}}=0$,\\
(ii) $g_{j}-g_{k}$ are not constants for $1\leq j<k\leq n$,\\
(iii) For $1\leq j\leq n$ and $1\leq h<k\leq n$,
$$T(r,f_{j})=S(r,e^{g_{j}-g_{k}})(r\rightarrow\infty, r\not\in E).$$
Then $f_{j}\equiv0$ for all $1\leq j\leq n$.
\end{lemma}

\begin{lemma}\label{22l}\cite{hly}
Suppose $f_{1},f_{2}$ are two nonconstant meromorphic functions on $\mathbb{C}^{n}$, then
$$N(r,f_{1}f_{2})-N(r,\frac{1}{f_{1}f_{2}})=N(r,f_{1})+N(r,f_{2})-N(r,\frac{1}{f_{1}})-N(r,\frac{1}{f_{2}}).$$
\end{lemma}\section{The proof of Theorem 1 }
Suppose
\begin{eqnarray}
\varphi=\frac{L(f)(f-g)}{(f-a_{1})(f-a_{2})},
\end{eqnarray}
and
\begin{eqnarray}
\psi=\frac{L(g)(f-g)}{(g-a_{1})(g-a_{2})}.
\end{eqnarray}
 Easy to know that $\varphi\not\equiv0$ because of $f\not\equiv g $, and $N(r,\varphi)=o(T(r,f))$.  By Lemma 3.1 and Lemma 3.5,  we have
\begin{eqnarray*}
\begin{aligned}
&T(r,\varphi)=m(r,\varphi)=m(r,\frac{L(f)(f-g)}{(f-a_{1})(f-a_{2})})+o(T(r,f))\notag\\
&=m(r,\frac{L(f)f}{(f-a_{1})(f-a_{2})}\frac{f-\sum_{i=0}^{n}b_{i}f^{(k_{i})}_{i\eta}}{f})+m(r,\frac{L(f)fb_{-1}}{(f-a_{1})(f-a_{2})})\\
&+o(T(r,f))\leq o(T(r,f)).
\end{aligned}
\end{eqnarray*}
That is
\begin{align}
T(r,\varphi)=o(T(r,f)).
\end{align}
Let $d=a_{1}-j(a_{1}-b_{1})(j\neq0,1)$. Obviously, by Lemma 3.1 and  Lemma 3.5, we obtain

\begin{align}
 &m(r,\frac{1}{f(z)})=m(r,\frac{L(f(z))f(z)}{(f(z)-a_{1})(f(z)-a_{2})}\frac{f(z)-\sum_{i=0}^{n}b_{i}f^{(k_{i})}(z+i\eta)}{f(z)})\notag\\
&+m(r,\frac{L(f(z))f(z)b_{-1}}{(f(z)-a_{1})(f(z)-a_{2})})+o(T(r,f))\notag\\
&\leq 2m(r,\frac{L(f(z))f(z)}{(f(z)-a_{1})(f(z)-a_{2})})+m(r,\frac{f(z)-\sum_{i=0}^{n}b_{i}f^{(k_{i})}(z)}{f(zi\eta)})\notag\\
&+o(T(r,f))=o(T(r,f)),
\end{align}
and hence
\begin{align}
 m(r,\frac{1}{f(z)-d})&=m(r,\frac{L(f(z))(f(z)-g(z))}{(\varphi(z) (f(z)-a_{1})(f(z)-a_{2})(f(z)-d)})\notag\\
&\leq m(r,\frac{L(f(z))f(z)}{(f(z)-a_{1})(f(z)-a_{2})(f(z)-d)})\notag\\
&+ m(r,1-\frac{g(z)}{f(z)})+o(T(r,f))=o(T(r,f)),
\end{align}

Set
\begin{align}
\phi=\frac{L(g)}{(g-a_{1})(g-a_{2})}-\frac{L(f)}{(f-a_{1})(f-a_{2})}.
\end{align}
We  discuss  two cases.\\

{\bf Case 1}\quad $\phi\equiv0$.  Integrating $\phi$ which leads to
$$\frac{g-a_{2}}{g-a_{1}}=C\frac{f-a_{2}}{f-a_{1}},$$
where $C$ is a nonzero constant.\\

If $C=1$, then $f=g$. If $C\neq1$, then from above, we have
$$\frac{a_{1}-a_{2}}{g-a_{1}}\equiv \frac{(C-1)f-Ca_{2}+a_{1}}{f-a_{1}},$$
and
$$T(r,f)=T(r,g)+o(T(r,f)).$$
It follows that $N(r,\frac{1}{f-\frac{Ca_{2}-a_{1}}{C-1}})=N(r,\frac{1}{a_{1}-a_{2}})=o(T(r,f))$. Then by Lemma 3.6,
\begin{eqnarray*}
\begin{aligned}
2T(r,f)&\leq \overline{N}(r,f)+\overline{N}(r,\frac{1}{f-a_{1}})+\overline{N}(r,\frac{1}{f-a_{2}})+\overline{N}(r,\frac{1}{f-\frac{Ca_{2}-a_{1}}{C-1}})+o(T(r,f))\\
&\leq \overline{N}(r,\frac{1}{f-a_{1}})+\overline{N}(r,\frac{1}{f-a_{2}})+o(T(r,f)),
\end{aligned}
\end{eqnarray*}
that is $2T(r,f)\leq \overline{N}(r,\frac{1}{f-a_{1}})+\overline{N}(r,\frac{1}{f-a_{2}})+o(T(r,f))$,
which contradicts with (3.11).

{\bf Case 2} \quad $\phi \not\equiv0$. By (3.11), (4.3) and (4.6), we can obtain
\begin{align}
m(r,f)&=m(r,f-g)+o(T(r,f))\notag\\
&=m(r,\frac{\phi(f-g)}{\phi})+o(T(r,f))=m(r,\frac{\psi-\varphi}{\phi})+o(T(r,f))\notag\\
&\leq T(r,\frac{\phi}{\psi-\varphi})+o(T(r,f))\leq T(r,\psi-\varphi)+T(r,\phi)+o(T(r,f))\notag\\
&\leq T(r,\psi)+T(r,\phi)+o(T(r,f))\notag\\
&\leq T(r,\psi)+\overline{N}(r,\frac{1}{f-a_{2}})+o(T(r,f)).
\end{align}
On the other hand,
\begin{align}
T(r,\psi)&=T(r,\frac{L(g)(f-g)}{(g-a_{1})(g-a_{2})})\notag\\
&=m(r,\frac{L(g)(f-g)}{(g-a_{1})(g-a_{2})})+o(T(r,f))\notag\\
&\leq m(r,\frac{L(g)}{g-a_{2}})+m(r,\frac{f-g}{g-a_{1}})\notag\\
&\leq m(r,\frac{1}{f-a_{1}})+o(T(r,f))=\overline{N}(r,\frac{1}{f-a_{2}})+o(T(r,f)),
\end{align}
hence combining  (4.7) and (4.8), we obtain
\begin{align}
 T(r,f)\leq 2\overline{N}(r,\frac{1}{f-a_{2}})+o(T(r,f)).
\end{align}
Next, Case 2 is  divided into two subcases.

{\bf Subcase 2.1}\quad $a_{1}=G$, where $G$ is defined as (3.2) in Lemma 3.7. Then by (3.10) and Lemma 3.1 we can get
\begin{align}
 m(r,Be^{h})=m(r,\frac{g-G}{f-a_{1}})=o(T(r,f)).
\end{align}
Then by (3.19), (4.8)  and (4.9) we can have $T(r,f)=o(T(r,f))$, and thus a contradiction.\\

{\bf Subcase 2.2} \quad $a_{2}=G$. Then by (3.16), (3.19), (4.9) and Lemma 3.1, we get
\begin{align}
 T(r,f)&\leq m(r,\frac{1}{f-a_{1}})+\overline{N}(r,\frac{1}{g-G})+o(T(r,f))\notag\\
 &\leq m(r,\frac{1}{g-G})+\overline{N}(r,\frac{1}{g-G})+o(T(r,f))\notag\\
 &\leq T(r,g)+o(T(r,f)).
\end{align}
From the fact that
\begin{align}
 T(r,g)\leq T(r,f)+o(T(r,f)),
\end{align}
which follows from (4.11) that
\begin{align}
 T(r,f)=T(r,g)+S(r,f).
\end{align}
By Lemma 3.1, Lemma 3.6, (3.11) and (4.13), we have
\begin{eqnarray*}
\begin{aligned}
2T(r,f)&\leq 2T(r,g)+o(T(r,f))\\
&\leq\overline{N}(r,\frac{1}{g-a_{1}})+\overline{N}(r,\frac{1}{g-G})+\overline{N}(r,\frac{1}{g-d})+o(T(r,f))\\
&\leq \overline{N}(r,\frac{1}{f-a_{1}})+\overline{N}(r,\frac{1}{f-a_{2}})+T(r,\frac{1}{g-d})-m(r,\frac{1}{g-d})+o(T(r,f))\\
&\leq T(r,f)+T(r,g)-m(r,\frac{1}{g-d})+o(T(r,f))\\
&\leq 2T(r,f)-m(r,\frac{1}{g-d})+o(T(r,f)).
\end{aligned}
\end{eqnarray*}
Thus
\begin{eqnarray}
m(r,\frac{1}{g-d})=o(T(r,f)).
\end{eqnarray}
From the First Fundamental Theorem of Nevanlinna, Lemma 3.1, Lemma 3.2, (4.1)-(4.2), (4.13)-(4.14) and the fact that $f$ is a transcendental  entire function of $\rho_{2}(f)<1$, we obtain
\begin{eqnarray*}
\begin{aligned}
m(r,\frac{f-d}{g-d})&\leq m(r,\frac{f}{g-d})+m(r,\frac{d}{g-d})+o(T(r,f))\\
&\leq T(r,\frac{f}{g-d})-N(r,\frac{f}{g-d})+o(T(r,f))\\
&=m(r,\frac{g-d}{f})+N(r,\frac{g-d}{f})-N(r,\frac{f}{g-d})\\
&+o(T(r,f))\leq N(r,\frac{1}{f})-N(r,\frac{1}{g-d})+o(T(r,f))\\
&=T(r,\frac{1}{f})-T(r,\frac{1}{g-d})+o(T(r,f))\\
&=T(r,f)-T(r,g)+o(T(r,f))=o(T(r,f)).
\end{aligned}
\end{eqnarray*}
Thus we get
\begin{eqnarray}
m(r,\frac{f-d}{g-d})=o(T(r,f)).
\end{eqnarray}
It's easy to see that $N(r,\psi)=o(T(r,f))$.  And we rewrite (4.2) as
\begin{eqnarray}
\psi=[\frac{a_{1}-d}{a_{1}-a_{2}}\frac{L(g)}{g-a_{1}}-\frac{a_{2}-d}{a_{1}-a_{2}}\frac{L(g)}{g-a_{2}}][\frac{f-d}{g-d}-1].
\end{eqnarray}
Then by  (4.15) and (4.16) we can get
\begin{eqnarray}
T(r,\psi)=m(r,\psi)+N(r,\psi)=o(T(r,f)).
\end{eqnarray}
By (4.13), (4.7), and (4.17) we get
\begin{eqnarray}
\overline{N}(r,\frac{1}{f-a_{1}})=o(T(r,f)).
\end{eqnarray}
Moreover, by (3.11), (4.13) and (4.18), we have
\begin{eqnarray}
m(r,\frac{1}{g-G})=o(T(r,f)),
\end{eqnarray}
which implies
\begin{eqnarray}
\overline{N}(r,\frac{1}{f-a_{2}})=m(r,\frac{1}{f-a_{2}})\leq m(r,\frac{1}{g-G})=o(T(r,f)).
\end{eqnarray}
Then by (3.11) we obtain $T(r,f)=o(T(r,f))$, and thus a contradiction.\\

{\bf Subcase 2.3} $a_{1}\not\equiv G, a_{2}\not\equiv G$. So by (3.16), (3.19), (4.9) and Lemma 3.6, we can get
\begin{eqnarray*}
\begin{aligned}
T(r,f)&\leq 2m(r,\frac{1}{f-a_{1}})+o(T(r,f))\leq2m(r,\frac{1}{g-G})\\
&+S(r,f)=2T(r,g)-2N(r,\frac{1}{g-G})+o(T(r,f))\\
&\leq\overline{N}(r,\frac{1}{g-a_{1}})+\overline{N}(r,\frac{1}{g-a_{2}})+\overline{N}(r,\frac{1}{g-G})\\
&-2N(r,\frac{1}{g-G})+o(T(r,f))\\
&\leq T(r,f)-N(r,\frac{1}{g-G})+o(T(r,f)),
\end{aligned}
\end{eqnarray*}
which deduces that
\begin{align}
N(r,\frac{1}{g-G})=o(T(r,f)).
\end{align}
It follows from Lemma 3.6 that
\begin{eqnarray*}
\begin{aligned}
T(r,g)&\leq \overline{N}(r,\frac{1}{g-G})+\overline{N}(r,\frac{1}{g-a_{1}})+o(T(r,f))\\
&\leq \overline{N}(r,\frac{1}{g-a_{1}})+o(T(r,f))\\
&\leq T(r,g)+o(T(r,f)),
\end{aligned}
\end{eqnarray*}
which implies that
\begin{align}
T(r,g)=\overline{N}(r,\frac{1}{g-a_{1}})+o(T(r,f)).
\end{align}
Similarly
\begin{align}
T(r,g)=\overline{N}(r,\frac{1}{g-a_{2}})+o(T(r,f)).
\end{align}
Then by (3.11) we get
\begin{align}
T(r,f)=2T(r,g)+o(T(r,f)).
\end{align}
Easy to see from (4.6) that
\begin{align}
T(r,\phi)=N(r,\phi)+S(r,f)\leq\overline{N}(r,\frac{1}{g-a_{2}})+S(r,f).
\end{align}
We claim that
\begin{align}
T(r,\phi)=\overline{N}(r,\frac{1}{g-a_{2}})+S(r,f).
\end{align}
Otherwise, 
\begin{align}
T(r,\phi)<\overline{N}(r,\frac{1}{g-a_{2}})+S(r,f).
\end{align}
We can  deduce from (3.11), Lemma 3.1, Lemma 3.5 and Lemma 3.12 that
\begin{eqnarray*}
\begin{aligned}
T(r,\psi)&=T(r,\frac{L(g)(f-g)}{(g-a_{1})(g-a_{2})})=m(r,\frac{L(g)(f-g)}{(g-a_{1})(g-a_{2})})+S(r,f)\notag\\
&\leq m(r,\frac{L(g)}{g-a_{1}})+m(r,\frac{f-a_{2}}{g-a_{2}}-1)\notag\\
&\leq m(r,\frac{g-a_{2}}{f-a_{2}})+N(r,\frac{g-a_{2}}{f-a_{2}})-N(r,\frac{f-a_{2}}{g-a_{2}})+S(r,f)\\
&\leq m(r,\frac{1}{f-a_{2}})+N(r,\frac{1}{f-a_{2}})-N(r,\frac{1}{g-a_{2}})+S(r,f)\\
&\leq T(r,f)-\overline{N}(r,\frac{1}{g-a_{2}})+S(r,f)\leq \overline{N}(r,\frac{1}{f-a_{1}})+S(r,f),
\end{aligned}
\end{eqnarray*}
which is
\begin{align}
T(r,\psi)\leq \overline{N}(r,\frac{1}{f-a_{1}})+S(r,f).
\end{align}
Then combining (3.11), (4.28) and the proof of (4.7), we obtain
\begin{eqnarray*}
\begin{aligned}
&\overline{N}(r,\frac{1}{f-a_{1}})+\overline{N}(r,\frac{1}{f-a_{2}})=T(r,f)+S(r,f)\notag\\
&\leq \overline{N}(r,\frac{1}{f-a_{1}})+T(r,\phi)+S(r,f),
\end{aligned}
\end{eqnarray*}
that is
\begin{align}
\overline{N}(r,\frac{1}{g-a_{2}})\leq T(r,\phi)+S(r,f),
\end{align}
a contradiction. Similarly, we can also obtain
\begin{align}
T(r,\psi)=\overline{N}(r,\frac{1}{g-a_{1}})+S(r,f).
\end{align}
By Lemma 3.7, if
\begin{align}
g=He^{p}+G,
\end{align}
where $H\not\equiv0$ is a small function of $e^{p}$.\\

Rewrite (4.2) as
\begin{align}
\phi\equiv\frac{L(g)(f-a_{1})(f-a_{2})-L(f)(g-a_{1})(g-a_{2})}{(f-a_{1})(f-a_{2})(g-a_{1})(g-a_{2})}.
\end{align}
Combing (3.1), (4.31) and (4.32),  we can set
\begin{align}
P&=L(g)(f-a_{1})(f-a_{2})-L(f)(g-a_{1})(g-a_{2})\notag\\
&=\sum_{i=0}^{5}\alpha_{i}e^{ip},
\end{align}
and
\begin{align}
Q&=(f-a_{1})(f-a_{2})(g-a_{1})(g-a_{2})\notag\\
&=\sum_{l=0}^{6}\beta_{l}e^{lp},
\end{align}
where $\alpha_{i}$ and $\beta_{l}$ are small functions of $e^{p}$, and $\alpha_{5}\not\equiv0$, $\beta_{6}\not\equiv0$.

If $P$ and $Q$ are two mutually prime polynomials in $e^{p}$, then by Lemma 3.8 we can get $T(r,\phi)=6T(r,e^{p})+o(T(r,f))$.  It follows from (3.19), (4.23)-(4.25) that $T(r,f)=o(T(r,f))$, and hence a contradiction.\\

If $P$ and $Q$ are  not two mutually prime polynomials in $e^{p}$, it's easy to see that the degree of $Q$ is large than $P$.\\
According to (4.31), (4.33), (4.34) and by simple calculation,  we must have
\begin{align}
\phi=\frac{C}{g-a_{2}},
\end{align}
where $C_{1}\not\equiv0$ is a small function of $f$.\\
Put (4.35) into (4.6) we have
\begin{align}
\frac{Cg-L(g)-Ca_{1}}{(g-a_{1})(g-a_{2})}\equiv\frac{-L(f)}{(f-a_{1})(f-a_{2})}.
\end{align}
By (4.6), we claim that $CHe^{p}\equiv (a_{1}-a_{2})(H'+p'H)e^{p}-(a'_{1}-a'_{2})He^{p}$. Otherwise, combining (4.6), (4.31),(4.36) and Lemma 3.10, we can get $T(r,e^{p})=o(T(r,f))$. It follows from  (3.19) and (4.9) that $T(r,f)=o(T(r,f))$, and hence a contradiction. Then substituting (4.31) into (4.6),  we have
\begin{align}
\psi=\frac{(CHe^{p}+F)(Ae^{p}-1)}{(He^{p}+G-a_{2})},
\end{align}
where $F=(G'-a'_{1})(a_{1}-a_{2})-(G-a_{1})(a'_{1}-a'_{2})$.  Put
$$R=ACHe^{2p}+(AF-CH)e^{p}-F,$$
$$S=He^{p}+G-a_{2}.$$
If $R$ and $S$ are two mutually prime polynomials in $e^{p}$, then by Lemma 3.8 we can get $T(r,\psi)=2T(r,e^{p})+o(T(r,f))$. Then by (3.19) and (4.8)-(4.9), we can get $T(r,f)=o(r,f)$. Therefore, $R$ and $S$ are not two mutually prime polynomials in $e^{p}$. (4.37) implies
\begin{align}
\psi=CAe^{p}, H=-A(G-a_{2}).
\end{align}

It follows from (4.37)-(4.38) that
\begin{align}
N(r,\frac{1}{CHe^{p}+F})=o(T(r,f)).
\end{align}

We claim that $F\equiv0$. Otherwise, if $F\not\equiv0$, then by (4.33), (4.34), and Lemma 3.6,
\begin{align}
T(r,e^{p})\leq \overline{N}(r,e^{p})+\overline{N}(r,\frac{1}{e^{p}})+\overline{N}(r,\frac{1}{e^{p}+F/CH})+o(T(r,f))=o(T(r,f)).
\end{align}
(3.19) and (4.8) deduce that $T(r,f)=o(T(r,f))$, and hence a contradiction.

Due to (4.31), (4.36) and (4.38), we can get
\begin{align}
H\equiv a_{2}A, \quad G\equiv0.
\end{align}
And hence
\begin{align}
g\equiv a_{2}Ae^{p},
\end{align}
\begin{align}
g-a_{2}=a_{2}(Ae^{p}-1).
\end{align}
Furthermore, we can deduce from (3.1) and (4.42) that
\begin{align}
f\equiv a_{2}A^{2}e^{2p}-a_{1}Ae^{p}+a_{1}.
\end{align}
Since $f$ and $g$ share $a_{2}$ IM, by (4.23)-(4.24) and (4.43)-(4.44) we obtain
\begin{align}
f-a_{2}&\equiv a_{2}A^{2}e^{2p}-a_{1}Ae^{p}+a_{1}-a_{2}\notag\\
&=a_{2}(Ae^{p}-1)^{2}.
\end{align}
It follows from $F\equiv0$, (4.44) and (4.45) that
\begin{align}
a_{1}\equiv2a_{2}.
\end{align}

By (4.46) and the fact that
 $$CHe^{p}\equiv (a_{1}-a_{2})(H'+p'H)e^{p}-(a'_{1}-a'_{2})He^{p},$$
 we get
\begin{align}
C\equiv \frac{A'}{A}+a_{2}p'.
\end{align}

it follows from (3.1), (4.36), (4.46) and (4.47) that
\begin{align}
A=a_{2}= 1, C=p'
\end{align}
and therefore
\begin{align}
a_{1}= 2.
\end{align}

\begin{align}
g(z)=e^{p},
\end{align}
where $c\neq0$ and $a$ are two finite constants.\\
Thus, by (3.1) and(4.48)-(4.50), we obtain
\begin{align}
f(z)=e^{2p}-2e^{p}+2.
\end{align}

If $m(r,e^{p})=m(r,e^{h})+O(1)=o(T(r,f))$. Then by (3.19) and (4.9), we deduce $T(r,f)=o(T(r,f))$, and thus a contradiction.\\

This completes the proof of Theorem 1.

\section{The proof of Corollary 1 }
By Theorem 1,   it reduces to the case that $f$ and $(\Delta_{\eta}^{n}f)^{(k)}$ share $2$ CM and $1$ IM. So in Lemma 3.6,
\begin{align}
L(f)=f', \quad L(g)=g'.
\end{align}
We claim that $p$ is a polynomial with $deg p=1$. Otherwise,  by (3.6) and (4.50),  we have
\begin{align}
\sum_{i=1}^{n}C_{i}e^{p(z+i\eta)-p(z)}-e^{Q}\equiv0,
\end{align}
where $e^{Q}=1$, and $C_{i}\not\equiv0(i=1,\ldots, n)$ are small functions of $e^{p(z+i\eta)-p(z)-Q}$ for all $i=1,\ldots, n$. Then by Lemma 3.11, we know that $C_{i}\equiv0$ for all $i=1,\ldots, n$, a contradiction. Hence,  according to  $c(\Delta_{\eta}^{n}f)^{(k)}\equiv(\Delta_{\eta}^{n}f)^{(k+1)}$, we know that
\begin{align}
f(z)=e^{2(cz+a)}-2e^{cz+a}+2.
\end{align}

It follows from  (3.6), (4.46) and (4.49) that
\begin{align}
-2(e^{c\eta}-1)^{n}=1.
\end{align}

It follows from above that
\begin{align}
e^{c\eta}=(-2)^{-\frac{1}{n}}+1.
\end{align}
But we can not get (3.2) from (5.5), a contradiction.

\section{proof of Corollary 2}
Suppose that $f\not\equiv g$.  Since $f$ and $g$ share $a_{1}$  and $a_{2}$ CM, we can get $T(r,\phi)=S(r,f)$ in {\bf Case 2} in Theorem 1. Then by (4.8) and (4.9), we know
\begin{align}
T(r,f)\leq \overline{N}(r,\frac{1}{f-a_{2}})+o(T(r,f)).
\end{align}
And hence
\begin{align}
T(r,f)&\leq \overline{N}(r,\frac{1}{f-a_{2}})+o(T(r,f))\notag\\
&=\overline{N}(r,\frac{1}{g-a_{2}})+o(T(r,f))\notag\\
&\leq T(r,g)+o(T(r,f)).
\end{align}
Then by (3.12), we have (3.13). According to a similar method of {\bf Subcases 2.2}, we can  obtain a contradiction.

\

{\bf Conflict of Interest}  The author declares that there is  no conflict of interest regarding the publication of this paper.

\

{\bf Acknowledgements} The author would like to thank to anonymous referees for their helpful comments.

\

{\bf Data availability statement}  There is no associated data in the paper. And the author declares that data will be shared with Research Square for the delivery of the author dashboard.



\begin{thebibliography}{99}
\bibitem{cao}  T. B. Cao,  \emph{Difference analogues of the second main theorem for meromorphic functions in several complex variables},
Math Nachr.  287 (2014), 530-545.

\bibitem{cx}  T. B. Cao,  L. Xu \emph{Logarithmic difference lemma in several complex variables
and partial difference equations},
Annali di Matematica Pura ed Applicata.  199 (2020), 767-794.

\bibitem{cy}  Z. X. Chen, H. X. Yi,  \emph{On Sharing Values of Meromorphic Functions and Their Differences},  Res. Math. 63 (2013),  557-565.


\bibitem{cf1}  Y. M. Chiang, S. J. Feng,
\emph{ On the Nevanlinna characteristic of $f(z+\eta) $ and
difference equations in the complex plane}, Ramanujan J. 16 (2008), no. 1,
105-129.

\bibitem{cf2} Y. M. Chiang, S. J. Feng, \emph{ On the growth of logarithmic differences, difference
quotients and logarithmic derivatives of meromorphic functions},
Trans. Amer. Math. Soc. 361 (2009), 3767-3791.



\bibitem{cc}   N. Cui, Z. X. Chen,  \emph{The conjecture on unity of meromorphic functions concerning their differences},  J. Diff. Equ. Appl. 22 (2013),  1452-1471.

\bibitem{g} G. G. Gundersen, \emph{ Meromorphic functions that share three or four values}, J. London Math. Soc. 20(1979), 457-466.

\bibitem{h1} R. G. Halburd, R. J. Korhonen,
\emph{ Difference analogue of the lemma on the logaritheoremic derivative with
applications to difference equations}, J. Math. Anal. Appl. 314 (2006),
no. 2, 477-487.

\bibitem{h2} R. G. Halburd, R. J. Korhonen,
\emph{ Nevanlinna theory for the difference operator},
Ann. Acad. Sci. Fenn. Math. 31 (2006), no. 2, 463-478.



\bibitem{h3} W. K. Hayman,\emph{ Meromorphic functions}, Oxford Mathematical Monographs Clarendon Press, Oxford 1964.

\bibitem{hkl} J. Heittokangas, R. Korhonen,  I. Laine, J. Rieppo, \emph{ Uniqueness of meromorphic functions sharing values with their shifts}, Complex Var. Elliptic Equ. 56 (2011), 81-92.

\bibitem{h}H$\ddot{o}$rmander, L., An introduction to complex analysis in several variables, Van Nostrand.
Princeton, N.J. 1966.

\bibitem{hy2}  P.C. Hu, C. C. Yang, \emph{Uniqueness of meromorphic functions on $\mathbf{C}^{m}$}, Complex
Variables 30(1996), 235-270.

\bibitem{hy3}  P.C. Hu, C. C. Yang, \emph{Further results on factorization of
meromorphic solutions of partial differential equations}, Results in Mathematics, 30 (1996), 310-320.

\bibitem{hly}  P.C. Hu, P. Li, C. C. Yang, \emph{Unicity of Meromorphic Mappings}, Berlin, Germany: Springer Science and Business Media,
2013.

\bibitem{h4} X.H. Huang, \emph{Unicity on entire function concerning its differential-difference operators}. Res. Math. Accepted.

\bibitem{h5} X.H. Huang, \emph{Unicity on entire function concerning its differential polynomials in Several complex variables}.  arXiv:2009.08066v6.

\bibitem{hf} X.H. Huang, M.L. Fang, \emph{Unicity of Entire Functions Concerning Their Shifts and Derivatives}. Comput. Methods Funct. Theory. Published (2021)



\bibitem{k} R. Korhonen, \emph{A difference Picard theorem for meromorphic functions of several variables}, Comput Methods Funct
Theory 12 (2012),  343-361.

\bibitem{ly2}  Li P,  Yang C C.   \emph{Value sharing of an entire function and its derivatives},
J. Math. Soc. Japan. 51 (1999),  781-799.


\bibitem{lr} X. L. Liu, R. Korhonen, \emph{ On the periodicity of transcendental entire functions},
 Bull. Aust. Math. Soc. 101 (2020), 453-465.

 \bibitem{lyf}   D. Liu, D. G. Yang , M. L. Fang, \emph{Unicity of entire functions concerning shifts and difference operators}, Abstr. Appl. Anal. 2014, 5 pp.


\bibitem{ll} F. L$\ddot{u}$, W. R. L$\ddot{u}$, \emph{ meromorphic functions sharing three values with their difference operators}, Comput. Methods Funct. Theory 17 (2017), no. 3,  395-403.


\bibitem{r} M. Ru, \emph{Nevanlinna Theory and Its Relation to Diophatine Approximation}, World Scientific Publishing
Co, Singapore, 2001.

\bibitem{ruy} L. A. Rubel, C. C. Yang, \emph{ Values shared by an entire function and its derivative},
 Lecture Notes in Math. Springer, Berlin, 599 (1977), 101-103.

\bibitem{v}  A. Vitter, \emph{The lemma of the logarithmic derivative in several complex variables}, Duke Math.
J. 44 (1977), 89-104.


\bibitem{y1} C. C. Yang, H. X. Yi,
\emph{ Uniqueness theory of meromorphic functions}, Kluwer Academic Publishers Group, Dordrecht, 2003.

\bibitem{y2} L. Yang, \emph{ Value Distribution Theory}, Springer-Verlag, Berlin, 1993.

\bibitem{y3} K, Yamanoi, \emph{ The second main theorem for small functions and related problems},  Acta Math. 192 (2004), no. 2, 225-294.


\bibitem{y3} Z. Ye, \emph{A sharp form of Nevanlinna¡¯s second main  theorem for several complex variables},
Math. Z. 222 (1996), 81-95.



















\end{thebibliography}
\end{document}